\documentclass[11pt]{amsart}
\usepackage{amsmath,amssymb,amsthm,latexsym,amscd}
\usepackage{indentfirst}
 \DeclareMathOperator{\grad}{grad}

 \newcommand{\dif}{\mathrm{d}}
 \newcommand{\R}{\mathbb{R}}
 \newcommand{\ROM}[1]{\mathrm{\uppercase\expandafter{\romannumeral#1}}}
  \theoremstyle{definition}
 \newtheorem{theorem}{Theorem}[section]
 \newtheorem{lemma}{Lemma}[section]
 
 \newtheorem{corollary}{Corollary}[section]
  
 \newtheorem{remark}{Remark}[section]
 \newtheorem{proposition}{Proposition}[section]

\title[Compact embedded hypersurfaces]{\textbf{Compact embedded hypersurfaces with constant higher order anisotropic mean curvatures}}
\author[Y. J. He]{Yijun He}\address{School of Mathematical Sciences,
Shanxi University, Taiyuan 030006, P. R. China.}
\thanks {The first author was partially supported by Youth Science Foundation
of Shanxi Province, China (Grant No. 2006021001).}
\email{heyijun@sxu.edu.cn}
 \author[H. Li]{Haizhong
Li}\address{Department of Mathematical Sciences, Tsinghua
University, Beijing 100084, P. R.
China.}\email{hli@math.tsinghua.edu.cn}
\thanks {The second author was partially supported by the grant No. 10531090 of the NSFC and by
SRFDP.}
 \author[H. Ma]{Hui Ma}\address{Department of Mathematical Sciences, Tsinghua
University, Beijing 100084, P. R.
China.}\email{hma@math.tsinghua.edu.cn}
 \thanks{The third author was partially supported by NSFC grant
No. 10501028 and NKBRPC No. 2006CB805905.}
\author[J. Q. Ge]{Jianquan Ge}\address{Department of Mathematical Sciences, Tsinghua
University, Beijing 100084, P. R.
China.}\email{gejq04@mails.tsinghua.edu.cn} 
\date{}
\begin{document}
\maketitle
\begin{abstract}
 Given a positive function $F$ on $S^n$ which satisfies a convexity
condition, for $1\leq r\leq n$, we define the $r$-th anisotropic
mean curvature function $H^F_r$ for hypersurfaces in
$\mathbb{R}^{n+1}$ which is a generalization of the usual $r$-th
mean curvature function. We prove that a compact embedded
hypersurface without boundary in $\R^{n+1}$ with
$H^F_r=\mbox{constant}$ is the Wulff shape, up to translations and
homotheties. In case $r=1$, our result is the anisotropic version of
Alexandrov Theorem, which gives an affirmative answer to an open
problem of F. Morgan.
\end{abstract}

\section{Introduction}
 Let $F\colon S^n\to\mathbb{R}^+$
be a smooth function which satisfies the following convexity
condition:
\begin{equation}\label{1}
(D^2F+FI)_x>0,\quad\forall x\in S^n,
\end{equation}
where $S^n$ is the standard unit sphere in $\R^{n+1}$, $D^2F$
denotes the intrinsic Hessian of $F$ on $S^n$ and $I$ denotes the
identity on $T_x S^n$, $>0$ means that the matrix is positive
definite. We consider the map
\begin{equation}\label{2}
\begin{array}
  {l}
  \phi\colon S^n\to\mathbb{R}^{n+1},\\
 x\to F(x)x+(\grad_{S^n}F)_x,
\end{array}
\end{equation}
its image $W_F=\phi(S^n)$ is a smooth, convex hypersurface in
$\mathbb{R}^{n+1}$ called the Wulff shape of $F$ (see \cite{BM},
\cite{clarenz}, \cite{M}, \cite{K-Pa1}, \cite{K-Pa2}, \cite{K-Pa3},
\cite{K-Pa4}, \cite{palmer}, \cite{Taylor}, \cite{winklmann}). When
$F\equiv1$, the Wulff shape $W_F$ is just $S^n$.

Now let $X\colon M\to\mathbb{R}^{n+1}$ be a smooth immersion of a
compact, orientable hypersurface without boundary. Let $\nu\colon
M\to S^n$ denote its Gauss map.

 Let $A_F=D^2F+FI$, $S_F=-\dif(\phi\circ\nu)=-A_F\circ\dif \nu$. $S_F$ is
called the $F$-Weingarten operator, and the eigenvalues of $S_F$ are
called anisotropic principal curvatures. Let $\sigma_r$ be the
elementary symmetric functions of the anisotropic principal
curvatures $\lambda_1, \lambda_2, \cdots, \lambda_n$:
$$\sigma_r=\sum_{i_1<\cdots<i_r}\lambda_{i_1}\cdots\lambda_{i_r}\quad (1\leq r\leq n).$$
We set $\sigma_0=1$. The $r$-th anisotropic mean curvature $H^F_r$
is defined by $H^F_r=\sigma_r/C^r_n$, also see Reilly \cite{reilly}.
$H^F=H^F_1$ is called the anisotropic mean curvature. If $F\equiv1$,
then $H^F_r=H_r$ is just the $r$-th mean curvature of hypersurfaces
which has been studied by many authors (see \cite{CL}, \cite{Li},
\cite{MR}, \cite{Ro}). Thus, the $r$-th anisotropic mean curvature
$H^F_r$ generalized the $r$-th mean curvature $H_r$ of hypersurfaces
in the $(n+1)$-dimensional Euclidean space $\R^{n+1}$.

For hypersurfaces in $\R^{n+1}$, we have the following classical
Alexandrov Theorem which was proved first by Alexandrov in \cite{A}
and later by Reilly in \cite{reilly2}, Montiel-Ros in \cite{MR} and
Hijazi-Montiel-Zhang in \cite{HMZ}:
\begin{theorem}\label{tm1.0} (Alexandrov Theorem)
Let $X\colon M\to\mathbb{R}^{n+1}$ be a compact hypersurface without
boundary embedded in Euclidean space. If $H=\mbox{constant}$,
 then $X(M)$ is a sphere.
\end{theorem}

Following from a modification of Reilly's proof, Ros showed in
\cite{Ro1} that the sphere is the only compact embedded hypersurface
without boundary with constant scalar curvature in $\R^{n+1}$, which
gave a partial answer to Yau's conjecture \cite{Yau}. Thereafter,
Ros \cite{Ro} extended his result to any $r$-th mean curvature, and
later, Montiel and Ros gave another proof in \cite{MR}. Explicitly,
they proved:
 \begin{theorem}
 \label{tm1.1}
  (\cite{MR},
\cite{Ro}) Let $X\colon M\to\mathbb{R}^{n+1}$ be a compact
hypersurface without boundary embedded in Euclidean space. If
$H_r=\mbox{constant}$ for some $r=1, \cdots, n$, then $X(M)$ is a
sphere.
\end{theorem}

In this paper, we prove the following anisotropic version
 of Theorem \ref{tm1.1}:
 \begin{theorem}
 \label{tm1.2} Let $X\colon M\to\mathbb{R}^{n+1}$ be a compact
hypersurface without boundary embedded in Euclidean space. If
$H^F_r=\mbox{constant}$ for some $r=1, \cdots, n$, then up to
translations, $X(M)=\rho W_F$, where $\rho=-1/H^F_1$ is a constant.
\end{theorem}
\begin{remark}
 For $n=1$, Morgan \cite{M} proved that Theorem \ref{tm1.2} still holds for a more general condition:
 $F$ is only a continuous norm on $\R^2$ and $X\colon M\to\mathbb{R}^2$ is a
 closed curve immersed in $\R^2$. In case $r=1$, Theorem \ref{tm1.2} is actually the
anisotropic version of Alexandrov Theorem, which gives an
affirmative answer to the following open problem proposed by Morgan
in the same paper: Whether an embedded equilibrium, i.e.
hypersurfaces with constant anisotropic mean curvature in Euclidean
space, must be the Wulff shape? We also note that M. Koiso stated
this conjecture in \cite{K}.
\end{remark}
\begin{remark}
  Theorem \ref{tm1.1} follows by choosing $F\equiv1$ in Theorem
  \ref{tm1.2}.
\end{remark}

\section{Preliminaries}
Let $X\colon M\to\mathbb{R}^{n+1}$ be a compact connected
hypersurface immersed in Euclidean space. Let $\nu\colon M\to S^n$
denote its Gauss map. Suppose there exists a point where all the
principal curvatures with respect to $\nu$ are positive. By the
positiveness of $A_F$, all the anisotropic principal curvatures are
positive at this point. Using the results of G{\aa}rding (\cite{G}),
we have the following lemma (cf. Montiel-Ros \cite{MR}):
\begin{lemma}\label{llem1}
Let $X\colon M\to\mathbb{R}^{n+1}$ be a compact connected
hypersurface without boundary. Suppose that there exists a point
where all the principal curvatures are positive. Assume $H^F_r>0$
holds on every point of $M$, then the same holds for $H^F_k$, $k=1,
\cdots, r-1$. Moreover
\begin{equation}
  \label{ll1}
  (H^F_k)^{(k-1)/k}\leq H^F_{k-1},\quad (H^F_k)^{1/k}\leq H^F_1, \quad k=1,
  \cdots, r.
\end{equation}
If $k\geq2$, the equality in the above inequalities happens only at
points where all the anisotropic principal curvatures are equal.
\end{lemma}

Let $\{e_1,\cdots,e_n\}$ be a local orthogonal frame of $X\colon
M\to\mathbb{R}^{n+1}$, then we have the structure equations:
 \begin{equation}\label{x}
\left\{
\begin{array}
  {l}
  \dif X=\sum_i\omega_ie_i\\
  \dif \nu=-\sum_{ij}h_{ij}\omega_je_i\\
  \dif
  e_i=\sum_j\omega_{ij}e_j+\sum_jh_{ij}\omega_j\nu\\
  \dif\omega_i=\sum_j\omega_{ij}\wedge\omega_j\\
  \dif\omega_{ij}-\sum_k\omega_{ik}\wedge\omega_{kj}=
  -\frac12\sum_{kl}R_{ijkl}\theta_k\wedge\theta_l
\end{array}\right.
\end{equation}
where $\omega_{ij}+\omega_{ji}=0$, $R_{ijkl}+R_{ijlk}=0$, and
$R_{ijkl}$ are the components of the Riemannian curvature tensor of
$M$ with respect to the induced metric $\dif X\cdot \dif X$.

Let $s_{ij}$ denote the coefficient of $S_F$ with respect to $\{e_1,
\cdots, e_n\}$, that is
\begin{equation}\label{xx}
  -\dif(\phi\circ\nu)=-A_F\circ\dif\nu=\sum_{i,j}s_{ij}\omega_je_i,
\end{equation}
where $\phi$ is defined in (\ref{2}).

We call the eigenvalues of $S_F$ to be anisotropic principal
curvatures, and denote them by $\lambda_1, \cdots, \lambda_n$. From
the positive definiteness of $A_F$, there exists a non-singular
matrix $C$ such that $A_F=C^TC$, so $S_F=-A_F\circ\dif\nu$ is
similar to the real symmetric matrix $-C\circ\dif\nu\circ C^T$.
Thus, the anisotropic principal curvatures are all real. Moreover,
if $\lambda_1=\cdots=\lambda_n$, we have $S_F=H^F_1I$, so
$-\dif(\phi\circ\nu)=H^F_1\dif X$ by (\ref{x}) and (\ref{xx}). Thus,
we have the following lemma (cf. \cite{HeL1}, \cite{HeL2},
\cite{palmer}):
\begin{lemma}\label{lemma3.5}
  Let $X\colon M\to\mathbb{R}^{n+1}$ be a compact hypersurface without
boundary. If
$\lambda_1=\lambda_2=\cdots=\lambda_n=\mbox{const}\neq0$, then up to
translations,
  $X(M)=\rho W_F$, where $\rho=-1/H^F_1$.
\end{lemma}

We define $s_{ijk}$ by
\begin{equation}
  \label{xxx}
  \dif
  s_{ij}+\sum_ks_{ik}\omega_{kj}+\sum_ks_{kj}\omega_{ki}=\sum_ks_{ijk}\omega_k.
\end{equation}

Taking exterior differentiation of (\ref{xx}) and using (\ref{x}),
we get
\begin{equation}\label{xv}
  s_{ijk}=s_{ikj}.
\end{equation}

\begin{lemma}\label{lemmax}
 Let $X\colon M\to\mathbb{R}^{n+1}$ be a compact hypersurface without
boundary. If $n\geq 2$ and
$\lambda_1=\lambda_2=\cdots=\lambda_n\neq0$, then
$\lambda_1=\lambda_2=\cdots=\lambda_n=\mbox{const}$, so up to
translations, $X(M)=\rho W_F$, where $\rho=-1/H^F_1$.
\end{lemma}
\begin{proof}
From (\ref{xv}) and $s_{ji}=H^F_1\delta_{ij}$, we have
$$e_i(H^F_1)=\sum_js_{jij}=\sum_js_{jji}=ne_i(H^F_1),\quad 1\leq
i\leq n.$$
  Therefore $\lambda_1=\lambda_2=\cdots=\lambda_n=H^F_1$ is a
  constant, then the conclusion follows from Lemma \ref{lemma3.5}.
\end{proof}

 We define
$F^*\colon \R^{n+1}\to\R$ to be (see \cite{BM}):
\begin{equation}\label{4}
  F^*(x)=\sup\{\dfrac{\langle x, z\rangle}{F(z)}|\; z\in S^n\},
\end{equation}
\begin{proposition}
  \label{npro1}
  Let $x\in\R^{n+1}\setminus\{0\}$, $y, z\in S^n$, then we have:
  \begin{enumerate}
    \item $\langle\phi(y), z\rangle\leq F(z)$, and the equality holds if and
    only if $y=z$;
    \item $\langle x, y\rangle\leq F^*(x)F(y)$, and the equality holds if and
    only if $x=F^*(x)\phi(y)$.
  \end{enumerate}
\end{proposition}
\begin{proof}
  Proof of (\romannumeral1). It is obvious that $\langle\phi(y), z\rangle\leq F(z)$ is
  equivalent to $\langle\phi(y)-\phi(z), z\rangle\leq 0$. The
  function $\Phi\colon S^n\times S^n\to\R$ defined by
  $$\Phi(y, z)=\langle\phi(y)-\phi(z), z\rangle$$
  is smooth, so it attained its maximum at some point $(y_0,
  z_0)$ because $S^n\times S^n$ is compact. By differentiating the
  function $\Phi(y, z)$ with respect to $y$ at the point $(y_0,
  z_0)$, we get
  $$\langle A_F\circ \dif y, z\rangle_{(y_0, z_0)}=0.$$
  Thus, from the
  positiveness of $A_F$, $z_0$ is orthogonal to $S^n$ at the point $y_0$, so, we must have $y_0=\pm z_0$.
  Notice that $\Phi(z_0, z_0)=0$, $\Phi(-z_0, z_0)=-F(z_0)-F(-z_0)<0$,
  the function $\Phi$ must attain its maximum 0 at the point
  $(z_0, z_0)$, so $\langle\phi(y), z\rangle\leq F(z)$.
 If $\langle\phi(y), z\rangle=F(z)$, then $\Phi$ obtains its
 maximum 0 at the point $(y, z)$, by the same reason we
 have $y=z$.

 Proof of (\romannumeral2). It is obvious that $\langle x, y\rangle\leq F^*(x)F(y)$ by
 the definition of $F^*$. Now we suppose that $\langle x,
 y\rangle=F^*(x)F(y)$, then the function $\langle x-F^*(x)\phi(y),
 y\rangle$ obtains its maximum 0 at the point $(x, y)$. So,
 differentiating it with respect to $y$, we get
 $$\langle x-F^*(x)\phi(y),
 \dif y\rangle=0.$$
 Thus, it follows that $x-F^*(x)\phi(y)$ is orthogonal to $S^n$ at $y$, that is, $x-F^*(x)\phi(y)=ky$ for some $k$. Then from $\langle x-F^*(x)\phi(y),
 y\rangle=0$, we have $x-F^*(x)\phi(y)=0$.
\end{proof}
\begin{proposition}\label{pro2}
  We have:
  \begin{enumerate}
    \item $F^*(x)>0, \forall x\in\R^{n+1}\setminus\{0\}$;
    \item $F^*(tx)=tF^*(x), \forall x\in\R^{n+1}, t>0$;
    \item $F^*(x+y)\leq F^*(x)+F^*(y), \forall x, y\in\R^{n+1}$, and the equality holds if and only if $x=0$, or $y=0$ or $x=ky$ for some $k>0$.
    \item $W_F=\{x\in\R^{n+1}|\; F^*(x)=1\}$.
  \end{enumerate}
\end{proposition}
\begin{proof}
  (\romannumeral1) and (\romannumeral2) follow from the definition of $F^*$.
  By the definition of $F^*$ and (\romannumeral2) of Proposition \ref{npro1}, we easily get (\romannumeral4). We now prove
  (\romannumeral3). Suppose $x, y\neq0$. Let $z\in S^n$ be such that $F^*(x+y)=\langle x+y, z\rangle/F(z)$, then we have
  $$F^*(x+y)=\langle x+y, z\rangle/F(z)=\langle x, z\rangle/F(z)+\langle y, z\rangle/F(z)\leq F^*(x)+F^*(y),$$
  with the equality holding if and only if $F^*(x)=\langle x, z\rangle/F(z)$ and $F^*(y)=\langle y, z\rangle/F(z)$. So, if the equality holds, then
  from (\romannumeral2) of Proposition \ref{npro1} we have
  $$x=F^*(x)\phi(z), \quad y=F^*(y)\phi(z).$$
  Thus, $x=F^*(x)/F^*(y)y$.
\end{proof}

From Proposition \ref{pro2}, for any
$x\in\mathbb{R}^{n+1}\setminus\{0\}$, we have $x/F^*(x)\in W_F$,
thus there exists a unique $\psi(x)\in S^n$ such that
$x=F^*(x)\phi(\psi(x))$. From the implicit function theorem and the
convexity of $F$, the function
$F^*\colon\mathbb{R}^{n+1}\setminus\{0\}\to\mathbb{R}^+$ and
$\psi\colon\R^{n+1}\setminus\{0\}\to S^n$ are smooth.

\section{$F$-focal point and $F$-cut point}

We define $d_F\colon \R^{n+1}\times\R^{n+1}\to\R$ to be $d_F(x,
y)=F^*(y-x)$, then we have $d_F(x, y)>0$ when $x\neq y$, $d_F(x,
x)=0$ and $d_F(x, z)\leq d_F(x, y)+d_F(y, z)$. Note that in general
$d_F(x, y)\neq d_F(y, x)$; and when $F\equiv1$, $d_F$ is just the
Euclidean distance function $d$.

For every $p\in\R^{n+1}$, let $\exp_p$ be the exponential map in
$\R^{n+1}$ at the point $p$, then $\exp_p(u)=p+u$. So, from the
definition of $d_F$, we have
\begin{equation}\label{n5}
  d_F(p, \exp_p(t\phi(Y)))=t, \ \mbox{for every $Y\in S^n$ and
  $t\in\R^+$}.
\end{equation}

 Now, let $X\colon M\to\R^{n+1}$ be a compact
embedded hypersurface without boundary, and $\nu$ be the unit inner
normal vector field of $M$. For convenience, we identify each point
$p\in M$ with its image $X(p)\in\R^{n+1}$.

For each $y\in\R^{n+1}$, define
\begin{equation}\label{7}
  d_F(M, y)=\inf\{d_F(p, y)|\; p\in M\}.
\end{equation}

 We define a function $c\colon M\to\R^+$ to be
such that $c(p)$ is the greatest $t\in(0, \infty)$ satisfying
$d_F(M, \exp_p(t\phi\circ (\nu(p))))=t$. We call
$\exp_p(c(p)\phi\circ \nu(p))$ the $F$-cut point of $p\in M$.

For $p\in M$, let $\gamma_p$ be the ray $\gamma_p\colon[0,
\infty)\to\R^{n+1}$ defined by:
$$\gamma_p(t)=p+t\phi\circ\nu(p),\quad \forall t\in [0, \infty),$$
and $\Gamma_p=\gamma_p([0, c(p)))$.
 Then we have
\begin{lemma}\label{lem4}
For $p, q\in M$, $p\neq q$, we have
$\Gamma_p\cap\Gamma_q=\emptyset$.
\end{lemma}
\begin{proof}
  Suppose $x\in\Gamma_p\cap\Gamma_q$, then there exists $0<t<\min(c(p),
  c(q))$ such that
  $$x=\exp_p(t\phi\circ\nu(p))=\exp_q(t\phi\circ\nu(q)),$$
  by the definition of $c(p), c(q)$ and (\ref{n5}).

  Suppose $t<s<c(p)$, then from (\romannumeral3) of Proposition \ref{pro2}, we have
  $$\begin{array}
    {rcl}
    d_F(q, \exp_p(s\phi\circ\nu(p)))&<&d_F(q, x)+d_F(x, \exp_p(s\phi\circ\nu(p)))\\
  &=&d_F(p, x)+d_F(x, \exp_p(s\phi\circ\nu(p)))\\
  &=&d_F(p, \exp_p(s\phi\circ\nu(p)))\\
  &=&d_F(M, \exp_p(s\phi\circ\nu(p))),
  \end{array}
  $$
  a contradiction.
\end{proof}

Consider the map: $\Psi\colon M\times\R\to\R^{n+1}$, $\Psi(p,
t)=\exp_p(t\phi\circ\nu(p))$. If $(p, t)$ is a critical point of the
map $\Psi$, then we call $\exp_p(t\phi\circ\nu(p))$ an $F$-focal
point of $p\in M$. Because $\exp_p(t\phi\circ\nu(p))=p+t\phi\circ
\nu(p)$, so through direct computation, we have
\begin{equation}\label{n11}
\dif(\exp_p(t\phi\circ\nu(p)))=(I-tS_F)\circ \dif p+(\phi\circ
\nu(p))\dif t.
\end{equation}
From (\ref{n11}), $\exp_p(t\phi\circ\nu(p))$ is an $F$-focal point
of $p$ if and only if the matrix $I-tS_F$ is degenerate. So, the
first $F$-focal point of $p$ along the ray $\gamma_p$ is
$\exp_p(1/\lambda_{\max}\phi\circ\nu(p))$, where $\lambda_{\max}$ is
the greatest positive anisotropic principal curvature at $p$.
\begin{remark}
  When $F=1$, $F$-cut point, $F$-focal point is the cut point and
  the  focal point of hypersurfaces in the Euclidean space
  respectively.
\end{remark}
\begin{lemma}\label{nnlem1}
  Either $(p, c(p))$ is a critical point of the map $\Psi$, or there
  exists at least one point $q\in M$, $q\neq p$, such that
  $d_F(q, \exp_p(c(p)\phi\circ\nu(p)))=c(p)$.
\end{lemma}
\begin{proof}
  We choose $\varepsilon_i>0$, such that
  $\lim_{i\to\infty}\varepsilon_i=0$. Let $a_i=\exp_p((c(p)+\varepsilon_i)\phi\circ
  \nu(p))$, $a=\exp_p(c(p)\phi\circ\nu(p))$. The continuity of $\Psi$
  implies that $\lim_{i\to\infty}a_i=a$. From the definition of
  $c(p)$, there exists points $q_i\in M$, such that
  $d_F(q_i, a_i)=d_F(M, a_i)=c(p)+\varepsilon'_i$, where
  $\varepsilon'_i<\varepsilon_i$, possibly $<0$. From the
  compactness of $M$, there exists a convergent subsequence of $\{q_i\}$, again denoted by $\{q_i\}$ such that
  $\lim_{i\to\infty}q_i=q$. Then we divided into two cases:

  {\it Case 1}. $q\neq p$. In this case we have
  $$\lim_{i\to\infty}d_F(q_i, a_i)=d_F(q, a),$$
  and
  $$\lim_{i\to\infty}d_F(M, a_i)=d_F(M, a)=c(p).$$
  So, we have $d_F(q, a)=c(p)$, as expected.

  {\it Case 2}. $q=p$. Suppose $(p, c(p))$ is not a critical point of the
  map $\Psi$, then there exists a neighborhood $U$ of $(p, c(p))\in
  M\times\R$ such that $\Psi|_U\colon U\to\Psi(U)$ is a
  diffeomorphism. And we have $\lim_{i\to\infty}d_F(q_i, a_i)=d_F(q, a)=c(p)$, so $\lim_{i\to\infty}\varepsilon'_i=0$.
  Therefore, for a sufficient large $i$, we have $(p,
  c(p)+\varepsilon_i), (q_i, c(p)+\varepsilon'_i)\in U$. But
  $$\exp_p((c(p)+\varepsilon_i)\phi\circ\nu(p))=\exp_{q_i}((c(p)+\varepsilon'_i)\phi\circ\nu(q_i)),$$
 thus we have
  $p=q_i$ and $\varepsilon_i=\varepsilon'_i$, a contradiction.
\end{proof}
\begin{lemma}\label{lem5}
  $c(p)\leq 1/\lambda_{\max}$,  where $\lambda_{\max}$ is the greatest positive anisotropic
  principal curvature at $p$.
\end{lemma}
\begin{proof}
  Let $t>1/\lambda_{\max}$. We define a function $h\colon M\to\R$ by
  \begin{equation}\label{11}
    h(q)=F^*(p+t\phi\circ
    \nu(p)-q), \forall q\in M.
  \end{equation}
  We prove that $p$ is not a local minimum point, so $t>c(p)$, thus the
  conclusion follows.

  Because, $(p+t\phi\circ
    \nu(p)-q)/h(q)\in W_F$, we can define a function $Y\colon M\to S^n$ by
  \begin{equation}\label{12}
    p+t\phi\circ\nu(p)-q=h(q)\phi\circ Y(q),
  \end{equation}
  and we have $Y(p)=\nu(p)$.

  Let $\gamma\colon (-\varepsilon, \varepsilon)\to M$ be a smooth curve such
  that $\gamma(0)=p$, we denote $h(s)=h(\gamma(s))$,
  $\gamma(s)=X(\gamma(s))$, $Y(s)=Y(\gamma(s))$, $\nu(s)=\nu(\gamma(s))$ for simplicity. Then, we
  have
  \begin{equation}\label{13}
 p+t\phi\circ\nu(p)-\gamma(s)=h(s)\phi\circ Y(s).
  \end{equation}
  Differentiating (\ref{13}), we get
  \begin{equation}\label{14}
    h'(s)\cdot\phi(Y(s))+h(s)\cdot (\phi(Y))'(s)=-\gamma'(s).
  \end{equation}
  From $\langle Y, Y\rangle=1$, we have $\langle Y'(s), Y(s)\rangle=0$,
  together with $(\phi(Y))'(s)=A_F\circ Y'(s)$ we have
  \begin{equation}\label{ac}
 \langle (\phi(Y))'(s), Y(s)\rangle=0.
  \end{equation}
  Thus, by taking inner product with $Y(s)$ in (\ref{14}), we get
  \begin{equation}
    \label{aa}
    h'(s)F(Y(s))=-\langle\gamma'(s), Y(s)\rangle.
  \end{equation}
  From (\ref{aa}) and $Y(0)=\nu(p)$, we have $h'(0)=-\langle\gamma'(0), \nu(p)\rangle/F(\nu(p))=0$, so $p$ is a extreme point of the function $h$. And we
  have
  \begin{equation}\label{16}
  t(\phi(Y))'(0)=tA_F\circ Y'(0)=-\gamma'(0),
  \end{equation}
Differentiating (\ref{14}), we get
  \begin{equation}\label{ab}
    h''(s)\cdot\phi(Y(s))+2h'(s)\cdot
    (\phi(Y))'(s)+h(s)(\phi(Y))''(s)=-\gamma''(s).
  \end{equation}
  Differentiating (\ref{ac}), we get
  \begin{equation}\label{19}
  \langle (\phi(Y))''(s), Y(s)\rangle=-\langle (\phi(Y))'(s), Y'(s)\rangle
  \end{equation}
  Thus, by taking inner product with $Y(s)$ in (\ref{ab}) and using (\ref{ac}), (\ref{19}), we get
  \begin{equation}\label{ad}
    h''(s)\cdot F(Y(s))-h(s)\langle A_F\circ Y'(s),
    Y'(s)\rangle=-\langle \gamma''(s), Y(s)\rangle.
  \end{equation}
  Evaluating (\ref{ad}) at $s=0$, using $Y(0)=\nu(p)$ and (\ref{16}) we have
  \begin{equation}\label{20}
    h''(0)\cdot F(\nu(p))-\frac1t\langle A_F^{-1}\circ \gamma'(0),
    \gamma'(0)\rangle=-\langle \gamma''(0), \nu(p)\rangle=\langle \gamma'(0), \nu'(0)\rangle.
  \end{equation}
Now, let $\gamma$ be such a curve that satisfies $-A_F\circ
  \nu'(0)=\lambda_{\max}\gamma'(0)$, that is, $\gamma'(0)$ is the
  eigenvector corresponding to the maximum positive eigenvalue of
  $S_F=-A_F\circ \dif\nu$. Then, we have
  \begin{equation}\label{21}
    h''(0)=\dfrac{\lambda_{\max}}{tF(\nu(p))}(\dfrac1{\lambda_{\max}}-t)\langle
    A_F^{-1}\circ\gamma'(0), \gamma'(0)\rangle<0,
  \end{equation}
  because $A_F^{-1}$ is positive definite.

  So, $p$ is not a local minimum point of the function $h$.
\end{proof}
\begin{lemma}\label{nnlem2}
  The map $c\colon M\to\R^+$ is continuous.
\end{lemma}
\begin{proof}
  Let $p_i\in M$ be such that $\lim_{i\to\infty}p_i=p$, we need to
  prove $\lim_{i\to\infty}c(p_i)=c(p)$. For any $q\in M$, we have
  $$d(q, \exp_q(c(q)\phi(\nu(q))))=c(q)\sqrt{|(\grad_{S^n}F)(\nu(q))|^2+[F(\nu(q))]^2}<\mbox{the diameter of $M$},$$
  so the function $c$ is bounded.

  Firstly, we prove $\limsup_{i\to\infty}c(p_i)\leq c(p)$. For any
  $\varepsilon>0$, there do not exist infinitely many indices $i$
  such that $c(p_i)>c(p)+\varepsilon$. Otherwise, by the definition of $c(p_i)$, we have
  $$d_F(p_i, \exp_p((c(p)+\varepsilon)\phi(\nu(p_i))))=c(p)+\varepsilon,$$
  and, by the continuity of the function $d_F$,
  $d_F(p, \exp_p((c(p)+\varepsilon)\phi(\nu(p))))=c(p)+\varepsilon$, which contradicts the definition of
  $c(p)$. Therefore $\limsup_{i\to\infty}c(p_i)\leq
  c(p)+\varepsilon$, for any $\varepsilon>0$, which proves the
  claim.

  Secondly, we prove $\liminf_{i\to\infty}c(p_i)\geq c(p)$. Let
  $\bar{t}=\liminf_{i\to\infty}c(p_i)$. Consider a subsequence of
  $\{c(p_i)\}$, again denoted by $c(p_i)$, which converges to
  $\bar{t}$. It is obvious an accumulation point of $F$-focal
  points is $F$-focal point, if for any such subsequence, the
  points $\exp_{p_i}(c(p_i)\phi(\nu(p_i)))$ are $F$-focal points of
  $p_i$, then $\exp_p(c(p)\phi(\nu(p)))$ is an $F$-focal point
  of $p$, hence $\bar{t}\geq c(p)$ by Lemma \ref{lem5}.

  Suppose, therefore, there exists a subsequence of $c(p_i)$ (again
  denoted by $c(p_i)$), such that $\exp_{p_i}(c(p_i)\phi(\nu(p_i)))$ is not $F$-focal point of $p_i$. By Lemma \ref{nnlem1}, there
  exists $q_i\in M$ such that $d_F(q_i, \exp_{p_i}(c(p_i)\phi(\nu(p_i))))=c(p_i)$. Taking, if necessary, a subsequence, we may
  suppose that $\lim_{i\to\infty}q_i=q\in M$. If $p\neq q$, by taking
  limit we see that, $d_F(q, \exp_p(\bar{t}\phi(\nu(p))))=d_F(p, \exp_p(\bar{t}\phi(\nu(p))))$, hence $\bar{t}\geq c(p)$. If
$p=q$, then for any
  neighborhood $V=U\times(\bar{t}-\varepsilon, \bar{t}+\varepsilon)$ of $(p,
\bar{t})$, there exists $i$,
  such that $p_i, q_i\in U$ and $c(p_i)\in (\bar{t}-\varepsilon, \bar{t}+\varepsilon)$.
  Choose $\bar{t}-\varepsilon<s<c(p_i)$, then we have $\exp_{p_i}(s\phi(\nu(p_i)))\neq\exp_{q_i}(s\phi(\nu(q_i)))$ by Lemma \ref{lem4},
  so the map $\Psi|_V\colon V\to\Psi(V)$ can not be injective.
  Thus, $(p, \bar{t})$ is a critical point of $\Psi$.
\end{proof}
\section{An integral inequality of compact hypersurfaces}
In this section we derive an integral inequality of compact
hypersurface without boundary embedded in Euclidean space (Theorem
\ref{thma}) which plays an important role in the proof of our main
theorem. First, we recall the following integral formulas of
Minkowski type for compact hypersurfaces in $\mathbb{R}^{n+1}$.

\begin{theorem}\label{th1}
 (\cite{HeL1}, \cite{HeL2})
Let $X\colon M\to\mathbb{R}^{n+1}$ be an $n$-dimensional compact
hypersurface without boundary, $F\colon S^n\to\mathbb{R}^+$ be a
smooth function which satisfies (1), then we have the following
integral formulas of Minkowski type:
\begin{equation}\label{3}
 \int_M(H^F_rF\circ\nu+H^F_{r+1}\langle X, \nu\rangle)\dif A=0,\quad r=0, 1,
 \cdots, n-1.
\end{equation}
\end{theorem}

Now, we let $X\colon M\to\R^{n+1}$ be a compact embedded
hypersurface without boundary, then $M$ is a boundary of some
compact domain $D\subset\R^{n+1}$, let $\nu$ be the unit inner
normal vector field of $M$.

\begin{lemma}\label{lem3}
  For any fixed point $y\in D\setminus X(M)$, there exists at least a point
  $p\in M$ such that
  \begin{equation}\label{8}
    y-p=d_F(M, y)\phi\circ\nu(p).
  \end{equation}
\end{lemma}
\begin{proof}
  From the compactness of $M$ and the continuity of the function
  $d_F$, there exists $p\in M$ such that $d_F(p, y)=\inf\{d_F(q, y)|\; q\in
  M\}$.

  Let $Z\colon M\to S^n$ be defined by
  $$y-q=F^*(y-q)\phi\circ Z(q).$$
  Then we have
  \begin{equation}\label{9}
 d_F(q, y)=\dfrac{\langle y-q, Z(q)\rangle}{F(Z(q))}.
  \end{equation}
  Differentiating (\ref{9}), we get
  \begin{equation}\label{10}
    \dif d_F(q, y)=-\dfrac{\langle \dif q, Z(q)\rangle}{F(Z(q))}.
  \end{equation}
  So, from the minimum of $p$, we get $\langle \dif p,
  Z(p)\rangle=0$, thus $Z(p) = \pm \nu(p)$. If $Z(p)=-\nu(p)$, then $\langle\phi\circ Z(p), \nu(p)\rangle=-F(Z(p))<0$, so the line segment
  connecting  $p$ and $y$ must intersect $X(M)$ at another point $\tilde{p}$, therefore $F^*(y-q)$ can not attain its minimum at $p$.
  Thus, $Z(p)=\nu(p)$ is the unit inner  normal vector.
\end{proof}
\begin{lemma}
  \label{lem4.2}
  Let $X\colon M\to\R^2$ be a simple closed curve and denote its arc
  parameter by $s$. Suppose $c(p)=1/\lambda(p)$ for some point $p\in
  M$, then we must have $\lambda'(p)=0$, where $\lambda$ is the
  anisotropic curvature and $'$ denote derivative with respect to
  the arc parameter.
\end{lemma}
\begin{proof}
  Let $x_0=p+c(p)\phi\circ\nu(p)$, define $\tau\colon M\to\R$ by:
  $$\tau(q)=F^*(x_0-q),$$
 then there exists a function $W\colon M\to S^1$ such that
 \begin{equation}
   \label{lem421}
   x_0-q=\tau(q)\phi\circ W(q), \quad W(p)=\nu(p).
 \end{equation}
 From the definition of $c(p)$, we have
 \begin{equation}\label{lem422}
   \tau(q)\geq c(p), \forall q\in M, \quad \tau(p)=c(p).
 \end{equation}
 Differentiating (\ref{lem421}), we get
 \begin{equation}
   \label{lem423}
   -T(q)=\tau'(q)\phi\circ W(q)+\tau(q)a(W(q))W'(q),
 \end{equation}
 where $T$ denotes the tangent vector of $M$, $a=D^2F+F$.

By taking inner product with $W(q)$ in (\ref{lem423}), we obtain
\begin{equation}
  \label{lem424}
  F(W(q))\tau'(q)=-\langle T(q), W(q)\rangle.
\end{equation}
Thus, we have $\tau'(p)=0$ by $W(p)=\nu(p)$. Then, from
(\ref{lem423}),
\begin{equation}
  \label{lem4241}
  W'(p)=-T(p)/(ac(p))=-\lambda(p)T(p)/a=-k(p)T(p)=\nu'(p),
\end{equation} where
$k$ is the curvature of $M$.

Differentiating (\ref{lem424}), we get
\begin{equation}
  \label{lem425}
  (F\circ W)'(q)\tau'(q)+F(W(q))\tau''(q)=-\langle T'(q),
  W(q)\rangle-\langle T(q), W'(q)\rangle.
\end{equation}
Evaluating (\ref{lem425}) at $q=p$, we get
$$F(\nu(p))\tau''(p)=-\langle k(p)\nu(p), \nu(p)\rangle-\langle T(p), -k(p)T(p)\rangle=0,$$
so $\tau''(p)=0$. Thus, differentiating (\ref{lem425}) and
evaluating it at $p$, we obtain
\begin{equation}
  F(\nu(p))\tau'''(p)=-\langle T''(p), \nu(p)\rangle-2\langle T'(p), W'(p)\rangle-\langle T(p),
  W''(p)\rangle.
\end{equation}

Differentiating (\ref{lem423}) and evaluating at $p$, we get
$$-T'(p)=c(p)((a\circ W)'(p)W'(p)+a\circ \nu(p)W''(p)).$$
As $W'(p)=\nu'(p)$, so $(a\circ W)'(p)=(a\circ\nu)'(p)$, through a
direct calculation, we have
\begin{equation}
  \label{lem4242}
  W''(p)=\frac{k(p)(a\circ\nu)'(p)}{a\circ\nu(p)}T(p)-k^2(p)\nu(p).
\end{equation}

From (\ref{lem4241}), (\ref{lem4242}) and $T''=-k^2T+k'\nu$, we
obtain
\begin{equation}
\label{lem426}
F(\nu(p))\tau'''(p)=-k'(p)-\frac{k(p)(a\circ\nu)'(p)}{a\circ\nu(p)}=-\frac{\lambda'(p)}{a\circ\nu(p)}.
\end{equation}
As $\tau(q)\geq\tau(p)$ holds for all $q\in M$, and
$\tau'(p)=\tau''(p)=0$, so we must have $\tau'''(p)=0$, thus
$\lambda'(p)=0$ as expected.
\end{proof}

 Let $f\colon D\to\R$ be an integrable function. From Lemma \ref{lem4}, Lemma \ref{nnlem2} and Lemma \ref{lem3}, we have
the following formula of integration
\begin{equation}\label{5}
  \int_DfdV=\int_M\int_0^{c(p)}f(\exp_p(t\phi(\nu(p))))E(p, t)\dif t\dif A,
\end{equation}
where $E(p, t)$ is given by
\begin{equation}\label{6}
  dV(\exp_p(t\phi(\nu(p))))=E(p, t)\dif t\dif A.
\end{equation}

If $x$ denotes the position vector in $\R^{n+1}$, we have
$\bar{\triangle}|x|^2=2(n+1)$, where $\bar{\triangle}$ is the
Euclidean Laplacian. From the Stokes Theorem, we have
\begin{equation}\label{22}
  -\int_M\langle X, \nu\rangle \dif A=(n+1)V,
\end{equation}
where $V$ the volume of $D$. From (\ref{n11}), we have
$$dV(p+t\phi(\nu(p)))=\det(I-tS_F)F\circ\nu \dif t\dif A=(1-t\lambda_1)\cdots(1-t\lambda_n)F\circ\nu \dif t\dif A.$$
Letting $f\equiv1$ in (\ref{5}) and taking into account that $E(p,
t)=(1-t\lambda_1)\cdots(1-t\lambda_n)F\circ\nu$, we have
\begin{equation}\label{23}
  V=\int_M\int_0^{c(p)}(1-t\lambda_1)\cdots(1-t\lambda_n)F\circ\nu \dif t\dif A.
\end{equation}

\begin{theorem}
  \label{thma}
  Let $X\colon M\to\R^{n+1}$ be a compact hypersurface without boundary embedded in
   Euclidean space. If the anisotropic mean curvature $H^F_1$ of
  $X$ with respect to the unit inner normal $\nu$ is everywhere positive on
  $M$, then we have
  \begin{equation}
    \label{rr}
    \int_M\dfrac{F\circ\nu}{H^F_1}\dif A\geq(n+1)V,
  \end{equation}
  where $V$ is the volume of the compact domain determined by $M$.
  Moreover, the equality holds in (\ref{rr}) if and only if  up to
  translations, $X(M)=\rho W_F$, where $\rho=-1/H^F_1$ is a constant.
\end{theorem}
\begin{proof}
  Firstly, if $X(M)=\rho W_F$, then $H^F_1=-1/\rho=\mbox{constant}$. So, by the
  integral equalities of Minkowski type (\ref{3}) and (\ref{22}),
  the equality in (\ref{rr}) holds.

  For $p\in M$, by Lemma \ref{lem5}, we have
  \begin{equation}\label{24}
    c(p)\leq 1/\lambda_{\max}\leq 1/H^F_1(p).
  \end{equation}

  Moreover, if $t\in[0, c(p))$, we have
  \begin{equation}\label{25}
 (1-t\lambda_1)\cdots(1-t\lambda_n)\leq(1-tH^F_1)^n,
  \end{equation}
  the equality holds only at points where $\lambda_1=\cdots=\lambda_n$. Thus, by putting
  (\ref{24}), (\ref{25}) into (\ref{23}), we get
  $$V\leq\int_M\int_0^{1/H^F_1}(1-tH^F_1)^nF\circ\nu \dif t\dif A=\dfrac1{n+1}\int_M\dfrac{F\circ\nu}{H^F_1}\dif A,$$
  and the equality holds if and only if $\lambda_1=\cdots=\lambda_n=1/c(p)$. Therefore,  by Lemma \ref{lemmax}, if
 $n\geq2$ and the equality holds, then up to
translations,
  $X(M)=\rho W_F$, where $\rho=-1/H^F_1$. If $n=1$, then from Lemma \ref{lem4.2}, $\lambda=H^F_1$ is a constant, so from Lemma \ref{lemma3.5},
  up to
translations,
  $X(M)=\rho W_F$, where $\rho=-1/H^F_1$.
\end{proof}
\begin{remark}
  By Lemma \ref{lemma3.5}, Theorem \ref{tm1.2} is true for $n=1$ even without the assumption of
  embedding. So, in order to prove Theorem \ref{tm1.2}, we actually don't need to prove the
  case $n=1$ of Theorem \ref{thma}. We prove it here only for completeness.
\end{remark}

If $F\equiv1$ in Theorem \ref{thma}, then we obtain
\begin{corollary}(\cite{MR}, \cite{Ro})
Let $X\colon M\to\R^{n+1}$ be a compact hypersurface without
boundary embedded in
  Euclidean space. If the mean curvature $H$ of
  $X$ with respect to the unit inner normal $\nu$ is everywhere positive on
  $M$, then we have
  $$\int_M\dfrac{1}{H}\dif A\geq(n+1)V,$$
  where $V$ is the volume of the compact domain determined by $M$.
  Moreover, the equality holds if and only if $X(M)$ is a round
  sphere.
\end{corollary}

\section{Proof of Theorem \ref{tm1.2}}
 We divide into two cases:

   {\it Case 1}. $\nu$ is the unit inner normal vector field.
Since $M$ is compact without boundary, one can find a point where
all the
   principal curvatures with respect to $\nu$ are positive.
   It follows from the positive definiteness of $A_F$ that
   all the anisotropic principal curvatures at this point with respect to $\nu$ are positive too.
    Thus $H^F_r$ is a positive constant. From Lemma \ref{llem1}, we have that
  $H^F_1, \cdots, H^F_{r-1}>0$, $(H^F_r)^{1/r}\leq H^F_1$ and
  $H^F_{r-1}\geq (H^F_r)^{(r-1)/r}$. Using Theorem \ref{thma}, we have
  \begin{equation}\label{26}
    (n+1)(H^F_r)^{1/r}V\leq \int_MF\circ\nu \dif A,
  \end{equation}
  and the equality holds if and only if up to
  translations, $X(M)=-\rho W_F$, where $\rho=1/H^F_1$ is a constant.

  Since $H^F_r$ is a positive constant and $(H^F_r)^{1/r}\leq H^F_1$, by Theorem \ref{th1}, we have
  $$\begin{array}
    {rcl}
    0 &=& \int_M(H^F_{r-1}F\circ\nu+H^F_r\langle X,
    \nu\rangle)\dif A\geq\int_M((H^F_r)^{(r-1)/r}F\circ\nu+H^F_r\langle X,
    \nu\rangle)\dif A\\
    &=& (H^F_r)^{(r-1)/r}\int_M(F\circ\nu+(H^F_r)^{1/r}\langle X, \nu\rangle) \dif A.
  \end{array}$$
  As $H^F_r$ is a positive constant, using (\ref{22}) we have
  $$\int_MF\circ\nu \dif A-(n+1)(H^F_r)^{1/r}V=\int_M(F\circ\nu+(H^F_r)^{1/r}\langle X,
  \nu\rangle)\dif A\leq0.$$
  Hence, the equality in (\ref{26}) holds, so up to
  translations, $X(M)=\rho W_F$, where $\rho=-1/H^F_1$ is a constant.

  {\it Case 2}. $\nu$ is the unit outer normal vector field.
  The conclusion follows as in Case 1 by considering the function
  $\tilde{F}\colon S^n\to\R^+$ defined by $\tilde{F}(x)=F(-x)$ instead of $F$.
  This completes the proof of Theorem \ref{tm1.2}.
\bigskip\\
{\bf Acknowledgements} The authors would like to thank Professor F.
Morgan for his helpful suggestions and comments on the original
version of this paper.

\bibliographystyle{amsplain}

\end{document}